\documentclass[12pt]{amsart}
\usepackage{amscd,verbatim}
\usepackage{amssymb}
\usepackage[colorlinks,linkcolor=blue,citecolor=blue,urlcolor=red]{hyperref}

\newcommand{\A}{\mathbf{A}}
\newcommand{\Q}{\mathbf{Q}}
\newcommand{\Z}{\mathbf{Z}}

\newcommand{\G}{\mathbb{G}}

\newcommand{\sF}{\mathcal{F}}

\newcommand{\fm}{\mathfrak{m}}

\renewcommand{\phi}{\varphi}

\newcommand{\inj}{\hookrightarrow}

\newcommand{\surj}{\rightarrow\!\!\!\!\!\rightarrow}
\newcommand{\Surj}{\relbar\joinrel\surj}
\newcommand{\by}{\xrightarrow}

\newcommand{\iso}{\by{\sim}}

\renewcommand{\lim}{\varprojlim}

\newcommand{\et}{{\operatorname{\acute{e}t}}}
\newcommand{\Nis}{{\operatorname{Nis}}}
\newcommand{\Zar}{{\operatorname{Zar}}}
\newcommand{\oc}{{\operatorname{oc}}}

\newcommand{\eff}{{\operatorname{eff}}}
\renewcommand{\o}{{\operatorname{o}}}

\newcommand{\car}{\operatorname{char}}

\newcommand{\Sm}{\operatorname{\mathbf{Sm}}}
\newcommand{\SLoc}{\operatorname{\mathbf{sLoc}}}

\newcommand{\Spec}{\operatorname{Spec}}

\newcommand{\uHom}{\operatorname{\underline{Hom}}}

\newcommand{\Ker}{\operatorname{Ker}}

\newcommand{\Coker}{\operatorname{Coker}}

\newcommand{\PST}{\operatorname{\mathbf{PST}}}

\newcommand{\HI}{\operatorname{\mathbf{HI}}}
\newcommand{\DM}{\operatorname{\mathbf{DM}}}

\numberwithin{equation}{section}

\newcounter{spec}
\newenvironment{thlist}{\begin{list}{\rm{(\roman{spec})}}%
{\usecounter{spec}\labelwidth=20pt\itemindent=0pt\labelsep=10pt}}%
{\end{list}}

\newtheorem{Th}{Theorem}
\swapnumbers
\newtheorem{thm}{Theorem}[section]
\newtheorem{prop}[thm]{Proposition}
\newtheorem{lemma}[thm]{Lemma}

\theoremstyle{definition}

\theoremstyle{remark}
\newtheorem{rk}[thm]{Remark}

\begin{document}

\title{Multiplicative properties of the multiplicative group}
\author{Bruno Kahn}
\address{IMJ-PRG\\Case 247\\
4 place Jussieu\\
75252 Paris Cedex 05\\
France}
\email{bruno.kahn@imj-prg.fr}
\date{November 2, 2017}

\begin{abstract}
We give a few properties equivalent to the Bloch-Kato conjecture (now the norm residue isomorphism theorem).\end{abstract}


\maketitle

\section*{Introduction} 

The Bloch-Kato conjecture, now called the norm residue isomorphism theorem, was finally proven by  Voevodsky in 2011 \cite{voeann}, using key inputs from Rost. The proof has many ramifications and involves a combination of sophisticated motivic techniques, including motivic Steenrod operations, and results of a more combinatorial kind like the existence of norm varieties. 

This state of the art gives some interest to the issue of finding a more elementary proof. In this direction, one can consider the early work of Thomason on inverting the Bott element in algebraic $K$-theory \cite{thomason} as a ``stable'' version of the conjecture; Levine later gave a motivic version of Thomason's theorem in \cite{levine}. I wondered how close to the norm residue isomorphism theorem the latter work takes us; the result is the following theorem, which was obtained in 2009.

\begin{Th}\label{T1} Let $k$ be an infinite perfect field and let $l$ be a prime number invertible in
$k$. If
$l=2$, assume that $k$ is non-exceptional in the sense of Harris-Segal: the Galois group of the
extension $k(\mu_{2^\infty})/k$ is torsion-free.  Then the following statements are equivalent:
\begin{thlist}
\item\label{i} The Beilinson-Lichtenbaum conjecture holds modulo $l$ over $k$.
\item\label{ii}  For all $n\ge 1$ and all $i>0$, $H_i(\G_m\otimes K_n^M/l)=0$. Here the tensor product is taken  in $\DM^\eff$.
\item\label{iii} For any $n\ge 2$, any function field $K/k$, any semi-local $K$-algebra $A$ and  any ideals $I,J\subset A$ with $I\cap J=0$, the map
\[K_n^M(A)/l\to K_n^M(A/I)/l\oplus K_n^M(A/J)/l\]
is injective.
\item\label{iv} Same as (iii), for $A$ the coordinate ring of $\hat{\Delta}^q_{K,S}$ for all $q\ge 2$ and all $\emptyset\neq S\subseteq [0,q]$ and $I,J$ defined by sets of vertices.
\end{thlist}
\end{Th}

Here are some explanations on the notation. We assume the reader familar with Voevodsky's category $\DM^\eff$ of effective motivic complexes \cite{voetri,mvw,bevo}; in (ii) and later, $H_i$ is relative to its homotopy $t$-structure. The Beilinson-Lichtenbaum conjecture is recalled at the end of \S \ref{s3}: it is equivalent to the Bloch-Kato conjecture by \cite{gl,sv}. If $A$ is a commutative semi-local ring, we write $K_*^M(A)$ for the Milnor ring of $A$ in the na\"\i ve sense, i.e. the quotient of the tensor algebra $T(A^*)$ by the two-sided ideal generated by elements $a\otimes (1-a)$ with $a,1-a\in A^*$. We shall write $K_n^M$ for the associated Nisnevich sheaf on the category $\Sm$ of smooth separated $k$-schemes of finite type. 

In (iv), we write $\hat\Delta_K^*$ for the
cosimplicial $K$-scheme whose $q$-th term $\hat\Delta_K^q$ is the semi-localisation of
$\Delta^q_K=\Spec K[t_0,\dots,t_q]/(\sum t_i = 1)$ at its vertices. If $i\in [0,q]$ (resp. $S\subseteq [0,q]$), we write $\hat\Delta^q_{K,i}$ for the $i$-th face of $\hat\Delta_K^q$ and $\hat\Delta^q_{K,S}=\bigcup_{i\in S} \hat\Delta^q_{K,i}$.  We shall also write $\partial_i$ for the inclusion $\hat\Delta^q_{K,i}\inj \hat\Delta^q_{K}$ ($i$-th face map), and $\hat\Delta^q_{K,[0,q]}=:\partial \hat\Delta^q_K$.

Of course, all statements in Theorem \ref{T1} are true since the first one is. The game we shall play here, however, is to forget about this fact and prove the equivalences without using it. Statement (ii) explains the title of this note. It is possible that such vanishing holds in more generality, which would be one possible direction of attack for a more elementary proof of \cite{voeann}. The scant evidence in this direction is a remarkable theorem of Sugiyama \cite[Prop. A.1]{ahew} that the tensor product of Nisnevich sheaves of $\Q$-vector spaces with transfers is exact. The most appealing leads are of course (iii) and (iv), because of their seemingly elementary nature. When I came up with Theorem \ref{T1}, I tried to prove either of these statements by using the techniques of  Guin and Nesterenko-Suslin in \cite{guin,ns}, but was not successful. 

(Added in November 2017.) When I sent this paper to Voevodsky in June 2017, he answered:

\begin{quote} \it I can not say that I knew this particular result, but I have encountered some facts of a similar nature and even tried to prove some of them. Without any success\dots\ It is strange that the existing proof is the only one known.

I am, BTW, partially in connection with my current interests, very interested in the elimination of the non-constructive elements from the proof of the BK or, at least, from the proof of the Merkurjev-Suslin theorem about $K_2/l.$

The main such element is the use of the axiom of choice or rather of the existence of well-ordering on any set quite early in the proof.

I am very interested in finding a proof that avoids this part of the argument.

(....)

I am sure that I can formalize constructively the statement of the BK. I can also formalize constructively most of my mathematics such as the motivic Steenrod operations.
\end{quote}

This was a few months before his death on September 30th, 2017. It will take time for many of us to recover from it.

\enlargethispage*{20pt}
\section{Proof of {\rm (i)} $\Rightarrow$ {\rm (iii)}}\label{s1}

Recall that the Bloch-Kato conjecture is a special case of the Beilinson-Lichtenbaum conjecture; the statement thus follows from:

\begin{prop}\label{p3.2} We assume the Bloch-Kato conjecture holds modulo $l$. Let $A$ be a semi-local $k$-algebra.  Let $I,J$ be two ideals of $A$ such that $I\cap J=0$. If $n\ge 2$, the homomorphism
\[K_n^M(A)/l\to K_n^M(A/I)/l\oplus K_n^M(A/J)/l\]
is injective.
\end{prop}

\begin{proof} By Kerz \cite[Th. 1.2]{kerz}, the norm residue homomorphism
\[K_n^M(A)/l\to H^n_\et(A,\mu_l^{\otimes n})\]
is bijective for $n\ge 1$. By the usual transfer argument \cite[Def. 5.5]{kerz}, we may assume that $\mu_l\subset k$.  Recall that \'etale cohomology with finite coefficients verifies closed Mayer-Vietoris, as a consequence of proper base change (for closed immersions!). Consider the diagram
\[\begin{CD}
K_{n-1}^M(A/I)\otimes \mu_l\oplus K_{n-1}^M(A/J)\otimes \mu_l@>>> H^{n-1}_\et(A/I,\mu_l^{\otimes n})\oplus H^{n-1}_\et(A/J,\mu_l^{\otimes n})\\
@VaVV @VVV\\
K_{n-1}^M(A/I+J)\otimes \mu_l@>>> H^{n-1}_\et(A/I+J,\mu_l^{\otimes n})\\
&& @V{\partial}VV\\
K_n^M(A)/l @>>> H^n_\et(A,\mu_l^{\otimes n})\\
@VbVV @VVV\\
K_n^M(A/I)/l\oplus K_n^M(A/J)/l@>>> H^n_\et(A/I,\mu_l^{\otimes n})\oplus H^n_\et(A/J,\mu_l^{\otimes n})
\end{CD}\]
where the horizontal maps are norm residue isomorphisms and $\partial$ is the boundary map for the long exact sequence corresponding to the closed covering $\Spec A=\Spec(A/I)\cup \Spec(A/J)$. The two squares obviously commute, and all horizontal maps are isomorphisms since $n\ge 2$. But $a$ is surjective, hence $\partial = 0$, hence $b$ is injective.
\end{proof}

\begin{rk} This proof does not work for $n=1$. In fact the conclusion is false: the short exact sequence
\[0\to A^*\to (A/I)^*\oplus (A/J)^*\to (A/I+J)^*\to 0\]
yields a long exact sequence
\begin{multline*}
0\to {}_lA^*\to {}_l(A/I)^*\oplus {}_l(A/J)^*\by{\rho} {}_l(A/I+J)^*\\
\to A^*/l\to (A/I)^*/l\oplus (A/J)^*/l\to (A/I+J)^*/l\to 0
\end{multline*}
so $\Coker\rho$ is finite but may be nontrivial if $A/I+J$ is too disconnected.
\end{rk}

\section{Motivic cohomology and Milnor $K$-theory}

For $n\ge 0$, the $n$-th motivic complex of Suslin and Voevodsky may be defined as
\[\Z(n) = C_*(\G_m^{\wedge n})[-n]\]
where $\G_m^{\wedge n}$ denotes the direct summand of $L((\A^1-0)^n)$ given by sections trivial at $(\A^1-0)^i\times\{1\} \times (\A^1-0)^{n-i-1}$ ($0\le i<n$) and $C_*$ is the Suslin complex \cite[Th. 15.2]{mvw}. We have the following basic results:

\begin{thm}[\protect{\cite{sv}, \cite[Th. 1.1]{kerz}}]\label{tsv} We have $\Z(0)=\Z$, $\Z(1)\simeq \G_m[-1]$, $H^i(\Z(n))=0$ for $i>n$ and $H^n(\Z(n))=K_n^M$.
\end{thm}

\section{Inverting the motivic Bott element, after Thomason and Levine}\label{s3}

Assume that $k$ contains a primitive $l$-th root of unity: the Nisnevich sheaf $\mu_l$ is then constant, cyclic of order $l$. From the exact triangle
\begin{equation}\label{eq6.3}
\mu_l[0]\to \Z/l(1)\to \G_m/l[-1]\by{+1}
\end{equation}
and the isomorphism $\Z/l(n)\otimes \Z/l(1)\iso \Z/l(n+1)$, we get a map in $\DM^\eff$:
\begin{equation}\label{eq2.1}
\Z/l(n)\otimes \mu_l\to \Z/l(n+1)
\end{equation}
hence another map
\[\Z/l(n)\to \Z/l(n+1)\otimes \mu_l^{-1}\]
which becomes an isomorphism after sheafifying for the \'etale topology. Let $i\le n$: iterating, we get a commutative diagram in $\HI$, the heart of the homotopy $t$-structure of $\DM^\eff$:

\[\begin{CD}
H^i(\Z/l(n)) &\to&
H^i(\Z/l(n+1)\otimes\mu_l^{-1})&\to&
H^i(\Z/l(n+2)\otimes\mu_l^{-2})&\to&\dots\\ 
@V{\psi_n^i}VV @V{\psi_{n+1}^i}VV @V{\psi_{n+2}^i}VV\\
H^i(R\alpha_*\alpha^*\Z/l(n)) &\iso&
H^i(R\alpha_*\alpha^*(\Z/l(n+1)\otimes\mu_l^{-1}))&\iso&
H^i(R\alpha_*\alpha^*(\Z/l(n+2)\otimes\mu_l^{-2}))&\iso&\dots 
\end{CD}\]
where $\alpha$ is the projection $\Sm_\et\to \Sm_\Nis$. We have:

\begin{thm}[\protect{\cite[Th. 1.1]{levine}}]\label{tlevine}  Assume that $k$ is non exceptional if $l=2$. Then the direct limit of the above diagram is a (vertical) isomorphism.
\end{thm}

(For $l=2$, Levine assumes either $\car k>0$ or that $k$ contains a square root of $-1$, but the
hypothesis he actually uses is that $k$ is not exceptional.)  

The Beilinson-Lichtenbaum conjecture is the statement that $\psi_n^i$ is an isomorphism for all $(i,n)$ such that $i\le n$. Hence Theorem \ref{tlevine} implies:

\begin{prop}\label{p1}  Under the assumption of Theorem \ref{tlevine}, the  Beilinson-Lichtenbaum conjecture holds modulo $l$ if and only if the map
\[H^i(\Z/l(n))\otimes\mu_l \to H^i(\Z/l(n+1))\]
is an isomorphism for any $(i,n)$ such that $i\le n$.\qed
\end{prop}

\section{Reformulation of Proposition \ref{p1}}

\begin{prop}\label{c5.1} a)  For all $n\ge 0$, the objects $\G_m\otimes K_{n}^M/l$ and
$\Z(n)[n]\otimes \G_m/l$ of $\DM^\eff$ are concentrated in cohomological degrees $\le
0$ (for the homotopy
$t$-structure), and we have isomorphisms
\[H^0(\G_m\otimes K_{n}^M/l)\simeq H^0(\Z(n)\otimes \G_m/l[n])\simeq K_{n+1}^M/l.\]
b) Assume that $k$ is non exceptional if $l=2$. Then the following statements are equivalent:
\begin{thlist}
\item The  Beilinson-Lichtenbaum conjecture holds modulo $l$.
\item For all $n\ge 1$, $\Z(n)\otimes \G_m/l\iso K_{n+1}^M/l[-n]$
in $\DM^\eff$.
\item  For all $n\ge 1$, $\G_m\otimes K_{n}^M/l\iso K_{n+1}^M/l[0]$ in $\DM^\eff$.
\item For all $n\ge 2$, the image of $K_n^M/l[0]$ under the localisation functor
$\nu_{\le0}:\DM^\eff\allowbreak\to
\DM^\o$ of \cite[(4.5)]{birat3} is $0$, where $\DM^\o$ is the category of birational motivic sheaves of \cite{birat3}.
\item For any function field $K/k$, any $n\ge 2$ and any $q\ge 0$, we have
\[H_q(K_n^M/l(\hat\Delta_K^*))=0.\]
\end{thlist}
\end{prop}

\begin{proof} a) follows from Theorem \ref{tsv}, the isomorphism $\Z(1)\simeq \G_m[-1]$ and the right $t$-exactness of $\otimes$ \cite[comment after (5.2)]{birat3}.
b) We reduce to $\mu_l\subset k$.  Let $C_n$ be the cone of \eqref{eq2.1}, so that $C_n\simeq \Z(n)\otimes \G_m/l[-1]$. In view of a) and Proposition \ref{p1}, (i) is equivalent to saying that $C_n$ is concentrated in degree $n+1$ and that the map
\[K_{n+1}^M/l\simeq H^{n+1}(\Z/l(n+1))\to H^{n+1}(C_n)\]
is an isomorphism. This shows that (i) $\iff$ (ii).

The identity
\[\Z(n)\otimes \G_m/l
\simeq \G_m\otimes \Z(n-1)\otimes (\G_m/l)[-1]\]
shows that (ii) $\iff$ (iii) by induction on $n$ (note that (ii) and (iii) are identical for $n=1$). 

By \cite[Prop. 4.2.5]{birat3}, the statement in (iv) is equivalent to $K_n^M/l$ being divisible by $\Z(1)$ in $\DM^\eff$,
which is implied by (iii). Conversely, if $K_n^M/l\simeq C(1)$ for some $C\in \DM^\eff$,
Voevodsky's cancellation theorem \cite{voecan} shows that $C\simeq
\uHom(\Z(1), K_n^M/l)=\uHom(\G_m,K_n^M/l)[1]=(K_n^M/l)_{-1}[1]=K_{n-1}^M/l[1]$ (compare \cite[Prop. 4.3 and Rk. 4.4]{ky}).

For (iv) $\iff$ (v), we use  \cite[Rk. 4.6.3]{birat3} (see also
\cite[Rk. 2.2.6]{azumaya}): let $i^\o:\DM^\o\to \DM^\eff$ be the inclusion. For any $C\in C(\PST)$ which is $\A^1$-invariant and satisfies Nisnevich excision, and for any connected $Y\in \Sm(k)$ with function field $K$, one has a quasi-iso\-morph\-ism
\begin{equation}\label{eq3.3}
(i^\o\nu_{\le 0} C_\Nis)(Y)\simeq R\Gamma(\hat\Delta_K^*,C).
\end{equation}

For any $\sF\in \HI$, one has $H^q_\Nis(X,\sF)=0$ for $q\ne 0$ for any smooth semi-local $k$-scheme $X$ as a consequence of \cite[Th. 4.37]{voepre}. Therefore, the right hand side of \eqref{eq3.3} for $C=\sF[0]$ is
quasi-isomorphic to the complex associated to the simplicial abelian group
\[\sF(\hat\Delta_K^*),\]
which shows the equivalence of (iv) and (v) by taking $\sF=K_n^M/l$.
This concludes the proof.
\end{proof}

\begin{rk} In Proposition \ref{c5.1} b), (ii) is also (trivially) true for $n=0$, but not (iii) (see
\eqref{eq6.3}).
\end{rk}

\section{Elementary lemmas on Milnor $K$-groups}

Let $A$ be a commutative semi-local ring, and let $I$ be an ideal of $A$. We write $(1+I)^*=(1+I)\cap A^*=\Ker(A^*\to (A/I)^*)$. 

\begin{lemma}\label{l3.1} Assume that $|A/\fm|>2$ for all maximal ideals $\fm$ of $A$. Then, with the above notation:
\begin{thlist}
\item $A^*\to (A/I)^*$ is surjective.
\item Let $\bar a\in A/I$ be such that $\bar a,1-\bar a\in (A/I)^*$. Then there exists $a\in A$ such that $a\mapsto \bar a$ and $a,1-a\in A^*$.
\item Let $J$ be another ideal of $A$, with image $\bar J\subset A/I$. Then $(1+J)^*\to (1+\bar J)^*$ is surjective.
\end{thlist}
\end{lemma}

\begin{proof} Let $R$ be the Jacobson radical of $A$, so that $1+R\subset A^*$. Assume first $R=0$: then $A$ is a finite product of fields and the three statements are obvious (the cardinality hypothesis is used in (ii)). The general case follows from chasing in the commutative square
\[\begin{CD}
A@>>> A/I\\
@VVV @VVV\\
A/R@>>> A/(R+I).
\end{CD}\]
\end{proof}

\begin{lemma}\label{l3.2} Keep the assumption of Lemma \ref{l3.1}. With the above notation, $K_*^M(A)\to K_*^M(A/I)$ is surjective with kernel the ideal generated by $(1+I)^*$.
\end{lemma}

\begin{proof} The first assertion follows from Lemma \ref{l3.1} (i). To prove the second one, let us construct a surjective section to the surjection
\[\frac{K_*^M(A)}{\{(1+I)^*\}K_*^M(A)}\Surj K_*^M(A/I).\]

It suffices to show that the surjective ring homomorphism $T((A/I)^*)\allowbreak\surj K_*^M(A)/\{(1+I)^*\}K_*^M(A)$ extending the identity map in degree $1$ kills the Steinberg relations: this follows from Lemma \ref{l3.1} (ii).
\end{proof}

\begin{prop}\label{p3.1} Keep the assumption of Lemma \ref{l3.1}, and let $I,J$ be two ideals of $A$. Then the sequence
\[K_*^M(A)\to K_*^M(A/I)\oplus K_*^M(A/J)\to K_*^M(A/I+J)\to 0\]
is exact.
\end{prop}

\begin{proof} Let $\bar I$ be the image of $I$ in $A/J$. Consider the commutative diagram
\[\begin{CD}
\{(1+I)^*\}K_*^M(A) @>>> K_*^M(A)@>>> K_*^M(A/I) @>>> 0\\
@VVV @VVV @VVV \\
\{(1+\bar I)^*\}K_*^M(A/J) @>>> K_*^M(A/J)@>>> K_*^M(A/I+J) @>>> 0.
\end{CD}\]

By Lemma \ref{l3.2}, the rows are exact and the middle and right vertical maps are surjective; by Lemma \ref{l3.1} (iii), the left vertical map is also surjective. The claim now follows from a diagram chase.
\end{proof}

\section{End of proof of Theorem \ref{T1}}

\begin{lemma}\label{l5.3} Let $\SLoc$ be the category of semi-local $K$-schemes. Let $F$ be a contravariant functor from $\SLoc$ to abelian groups. Suppose that, for any $X\in \SLoc$ and any closed cover $X=Z_1\cup Z_2$, the sequence
\[0\to F(X)\to F(Z_1)\oplus F(Z_2)\to F(Z_1\cap Z_2)\]
is exact. (Here, $Z_1\cap Z_2:=Z_1\times_X Z_2$ is the scheme-theoretic intersection.) Then, for any closed cover $X=Z_1\cup\dots \cup Z_r$, the sequence
\[0\to F(X)\to \bigoplus_{j=1}^rF(Z_j)\to \bigoplus_{j<k}F(Z_j\cap Z_k)\]
is exact.
\end{lemma}

\begin{proof} Of course this lemma is much more general and the point is to spell out its proof. Let $Y=Z_1\cup\dots Z_{r-1}$. By hypothesis, the sequence
\[0\to F(X)\to F(Y)\oplus F(Z_r)\to F(Y\cap Z_r)\]
is exact and, by induction on $r$, the map
\[F(Y\cap Z_r)\to \bigoplus_{j<r} F(Z_j\cap Z_r)\]
is injective. The conclusion follows by chasing in the diagram
\[\begin{CD}
0@>>> F(X)@>>>\displaystyle \bigoplus_{j=1}^rF(Z_j)@>>>\displaystyle \bigoplus_{j<k}F(Z_j\cap Z_k)\\
&&@VVV @VVV @VVV\\
0@>>> F(Y)@>>>\displaystyle \bigoplus_{j=1}^{r-1}F(Z_j)@>>>\displaystyle \bigoplus_{j<k<r}F(Z_j\cap Z_k).
\end{CD}\]
\end{proof}

\begin{lemma}[See also \protect{\cite[Lemma 2.4]{levine2}}]\label{l6.1}
Let $A_*=(A_n)_{n\ge 0}$ be a simplicial abelian group. Let $(A_n^0)$ be the normalised complex of $A$: $A_r^0= \bigcap_{i>0} \Ker (\partial_i:A_r\to A_{r-1})$. For $q>0$, consider the commutative diagram of complexes
\begin{equation}\label{eq6.1}
\begin{CD}
A_{q+1}^0 @>\partial_0>> A_q^0 @>\partial_0>> A_{q-1}^0\\
@V\alpha VV @V\beta VV @V\gamma VV \\
A_{q+1}@>a>> \displaystyle\bigoplus_{i=0}^{q} A_q@>{b}>>  \displaystyle\bigoplus_{0\le j<k\le q} A_{q-1}
\end{CD}
\end{equation}
where $\alpha$ is inclusion, $\beta(x)=(x,0,\dots,0)$,  $\gamma(x)_{j,k}= \begin{cases} x&\text{si $(j,k)=(0,1)$}\\ 0& \text{else,} \end{cases}$ $a=(\partial_i)_{0\le i\le q}$ and the $(i,j,k)$-component of $b$ is
\[\begin{cases}
0 &\text{if $i\ne j,k$}\\
\partial_{k-1} &\text{if $i=j$}\\
-\partial_{j} &\text{if $i=k$.}
\end{cases}\] 
Then this diagram induces an injection on homology.
\end{lemma}

\begin{proof} Obvious.
\end{proof}

\begin{proof}[End of proof of Theorem \ref{T1}] We saw in \S \ref{s1} that (i) $\Rightarrow$ (iii); we have  (i) $\iff$ (ii) by the equivalence (i) $\iff$ (iii) in Proposition \ref{c5.1} b). Obviously, (iii) $\Rightarrow$ (iv). It remains to show that (iv) $\Rightarrow$ (i).

Suppose that (iv) holds in Theorem \ref{T1}. In view of Proposition \ref{p3.1} and (the proof of) Lemma \ref{l5.3}, we get for all $q>0$ an exact sequence
\[0\to K_n^M(\partial\hat\Delta_K^q)/l\to \bigoplus_{j=0}^qK_n^M(\hat\Delta_{K,j}^q)/l\to \bigoplus_{j<k}K_n^M(\hat\Delta_{K,j,k}^q)/l.\]

But $K_n^M(\hat \Delta^q_K)\to K_n^M(\partial\hat\Delta_K^q)$ is surjective, hence the bottom row of \eqref{eq6.1} is exact for $A_*=K_*^M(\hat\Delta_K^*)/l$. By Lemma \ref{l6.1}, Condition (v) of Proposition \ref{c5.1} b) holds, and therefore so does its Condition (i). This concludes the proof. 
\end{proof}

\begin{rk}\label{r3.1} For $A_*=K_*^M(\hat\Delta_K^*)/l$, the homology group
of the bottom row of \eqref{eq6.1} may be reinterpreted in a more suggestive way: it is
\[\Coker\left(H^0_\Zar(\hat\Delta_K^{i+1},\sF)\to 
H^0_\oc(\partial\hat\Delta_K^{i+1},\sF)\right)\]
where $\oc$ denotes the open-closed topology introduced in \cite{glrev}.
\end{rk}

\end{document}